\documentclass[12pt]{amsart}
\usepackage{amsthm,amsmath,amssymb,mathtools}
\usepackage{geometry}
\usepackage[hypertexnames=false,hidelinks]{hyperref}
\usepackage{bookmark}
\usepackage[numbers,sort&compress]{natbib}
\usepackage{booktabs}
\usepackage{enumitem}
\usepackage{microtype}
\usepackage{tikz}
\usepackage{xcolor}
\newif\ifshowchanges
\showchangestrue

\makeatletter
\@namedef{subjclassname@2020}{\textup{2020} Mathematics Subject Classification}
\makeatother

\newtheorem{theorem}{Theorem}[section]
\newtheorem{lemma}[theorem]{Lemma}
\newtheorem{proposition}[theorem]{Proposition}

\newtheorem{remark}[theorem]{Remark}

\newcommand{\R}{\mathbb{R}}

\renewcommand{\epsilon}{\varepsilon}

\begin{document}

\title{Moderate deviations for the Maki--Thompson rumour model}

\author{Shaochen WANG}
\address[S. Wang]{School of Mathematics, South China University of Technology, Guangzhou, China}
\email{mascwang@scut.edu.cn}

\author{Guangyu YANG}
\address[G. Yang]{School of Mathematics and Statistics, Zhengzhou University, Zhengzhou, China}
\email{guangyu@zzu.edu.cn}

\date{\today}

\begin{abstract}
We consider the classical Maki--Thompson rumour model with $n$ initially ignorant individuals and one initial spreader. The final number of ignorants, divided by $n$, is known to satisfy the law of large numbers, the central limit theorem and the large deviation principle. We establish a moderate deviation principle for fluctuations of this proportion of order $b_n/\sqrt n$, where $b_n\to\infty$ and $b_n^2=o(\log n)$. The proof rests on the exact final-size distribution, sharp asymptotics for the associated automata numbers, and a uniform point probability expansion at the moderate deviation scale. Numerical comparisons of both tails illustrate the intermediate scale and the effect of early extinction on upper tail probabilities.
\end{abstract}


\maketitle

%%%%%%%%%%%%%%%%%%%%%%%%%%%%%%%%%%%%%%%%%%%%%%%%%%%%%%%%

\section{Introduction and main result}\label{sec:intro}

%%%%%%%%%%%%%%%%%%%%%%%%%%%%%%%%%%%%%%%%%%%%%%%%%%%%%%%%
%%%%%%%%%%%%%%%%%%%%%%%%%%%%%%%%%%%%%%%%%%%%%%%%%%%%%%%%

\subsection{The model}

%%%%%%%%%%%%%%%%%%%%%%%%%%%%%%%%%%%%%%%%%%%%%%%%%%%%%%%%

The Maki--Thompson model \cite{MakiThompson} is a fundamental stochastic model for rumour propagation in a closed homogeneous population. For broader accounts of epidemic and rumour processes, including models on networks, see \cite{DaleyGani,DraiefMassoulie10}. Consider a population of $n+1$ individuals. At time $t=0$, one individual knows the rumour (the \emph{spreader}) and the remaining $n$ individuals are ignorant of it. As time evolves, individuals can be in one of three states:

\begin{itemize}
    \item \emph{Ignorants}: unaware of the rumour;
    \item \emph{Spreaders}: know the rumour and are willing to spread it;
    \item \emph{Stiflers}: know the rumour but have ceased spreading it.
\end{itemize}

The dynamics are Markovian with the following two types of transitions: (1) \emph{Spreader--Ignorant meeting}: When a spreader meets an ignorant, the ignorant becomes a spreader. This occurs at rate proportional to the product of the number of spreaders and ignorants. (2) \emph{Spreader--Spreader or Spreader--Stifler meeting}: When a spreader meets another spreader or a stifler, the initiating spreader becomes a stifler. This occurs at rate proportional to the product of the number of spreaders and the number of other informed individuals.

More mathematically, let $X(t)$, $Y(t)$, $Z(t)$ denote respectively the numbers of ignorants, spreaders, and stiflers at time $t$. Initially, $X(0)=n$, $Y(0)=1$, $Z(0)=0$, and the total population size is conserved: $X(t)+Y(t)+Z(t)=n+1$ for all $t\ge 0$. The process stops at the random time $\tau = \inf\{t\ge 0: Y(t)=0\}$ when there are no spreaders left. We are interested in the final number of ignorants $X_n := X(\tau)$, i.e., the individuals who never hear the rumour. The process $\{(X(t),Y(t))\}_{t\ge0}$ is a continuous-time Markov chain on the state space $\{(i,j): 0\le i\le n,\ 0\le j\le n+1-i\}$. The states with $j=0$ are absorbing. Its nonzero transition rates are:
\begin{align}
    (i,j) &\to (i-1,j+1) \quad \text{at rate } ij, \label{11} \\
    (i,j) &\to (i,j-1) \quad \text{at rate } j(n-i). \label{12}
\end{align}
Transition (\ref{11}) corresponds to a spreader telling an ignorant; transition (\ref{12}) corresponds to a spreader becoming a stifler upon meeting another spreader or a stifler. The initiating spreader cannot contact itself, leaving $n-i$ other informed individuals.

%%%%%%%%%%%%%%%%%%%%%%%%%%%%%%%%%%%%%%%%%%%%%%%%%%

\subsection{Known asymptotic results}

%%%%%%%%%%%%%%%%%%%%%%%%%%%%%%%%%%%%%%%%%%%%%%%%%%

We now recall the main asymptotic results for the Maki--Thompson model, which serve as the foundation for our moderate deviation analysis. The asymptotic behaviour of $X_n$ as $n\to\infty$ has been studied extensively. Sudbury \cite{Sudbury85} proved the law of large numbers:
\begin{equation}
    \frac{X_n}{n} \xrightarrow{p} x_\infty \quad \text{as } n\to\infty,
\end{equation}
where $x_\infty$ is given by
\begin{equation}
    x_\infty = -\frac{W_0(-2e^{-2})}{2} \approx 0.203187869979\ldots,
\end{equation}
Here $W_0$ is the principal real branch of the Lambert $W$ function; see \cite{CorlessEtAl96}. It is characterized by
\[
 W_0(u)e^{W_0(u)}=u,\qquad W_0(u)\ge-1,\qquad u\ge-e^{-1}.
\]
Equivalently, $x_\infty$ is the unique solution in $(0,1/2)$ of $x e^{2(1-x)}=1$.

Watson \cite{Watson88} established the central limit theorem (CLT):
\begin{equation}
    \sqrt{n}\left(\frac{X_n}{n} - x_\infty\right) \xrightarrow{d} N(0,\sigma^2),
\end{equation}
where the asymptotic variance is
\begin{equation}\label{eq-sigma2}
    \sigma^2 = \frac{x_\infty(1-x_\infty)}{1-2x_\infty} \approx 0.272736.
\end{equation}

The law of large numbers has also been studied for models in which a spreader stops after $k$ unsuccessful contacts. Duan and Ganesh \cite{DuanGanesh21} and Lebensztayn and Santos \cite{LebensztaynSantos25} give laws of large numbers for such models. Recent work considers the maximum proportion of spreaders \cite{LeRo2025} and quasi-stationary distributions for the Maki--Thompson model and a variant with modified absorption \cite{BenAriEtAl25}.

Lebensztayn \cite{Lebensztayn15} established a large deviation principle (LDP) for $X_n/n$ with speed $n$. With the continuous endpoint convention at $1$, the corresponding rate function $H:\R_+\to[0,\infty]$, where $\R_+=[0,\infty)$, is
\begin{equation}
 H(x)=\begin{cases}
 h(x),&0\le x<1,\\
 0,   &x=1,\\
 +\infty,&x>1.
 \end{cases}
\end{equation}
Here $h(x)=x\log x+(1-x)[\varrho-\log(1-x)]$, with $0\log0=0$ and
$\varrho=2+\log[x_\infty(1-x_\infty)]\approx0.1792$.
The function $h$ is positive on $[0,1)\setminus\{x_\infty\}$ and tends to zero as $x\uparrow1$.
The vanishing rate at $1$ motivates a separate estimate of the endpoint contribution from the exact distribution.
Moreover, $h(x_\infty)=h'(x_\infty)=0$ and $h''(x_\infty)=1/\sigma^2$.
For general background on large deviations we refer to \cite{DemboZeitouni}; for rumour models see \cite{APRT,DaleyGani,DaleyKendall65,DPZMV,FHA2025,FJMRM2022}, \cite{LebensztaynMachadoRodriguez11,LebensztaynMachadoRodriguezRandom11,LeRo2025}.

%%%%%%%%%%%%%%%%%%%%%%%%%%%%%%%%%%%%%%%%%%%%%%%%%%%%%%

\subsection{Moderate deviations}

%%%%%%%%%%%%%%%%%%%%%%%%%%%%%%%%%%%%%%%%%%%%%%%%%%%%%%

The law of large numbers describes typical behaviour, the CLT describes small Gaussian fluctuations on the scale $1/\sqrt{n}$, and the LDP describes large deviations on the scale $1$. A moderate deviation principle (MDP) describes fluctuations on intermediate scales. For general theory of moderate deviations, see \cite[Section~3.7]{DemboZeitouni}, \cite{Wu95,Puhalskii01}.

Let $\{b_n\}_{n\ge1}$ be a positive sequence.  The local moderate deviation estimates (see Proposition \ref{prop:local} below) require only
\begin{equation}\label{bn-weak}
    b_n\to\infty,\qquad \frac{b_n}{\sqrt n}\to0.
\end{equation}
For the full MDP on $\mathbb R$, control of the right endpoint leads to the additional condition
\begin{equation}\label{bn-endpoint}
    \frac{b_n^2}{\log n}\to0.
\end{equation}
When both \eqref{bn-weak} and \eqref{bn-endpoint} hold, we shall write simply that $b_n$ satisfies
\begin{equation}\label{bn-condition}
    b_n\to\infty,\qquad \frac{b_n^2}{\log n}\to0.
\end{equation}

The goal of this note is to study the MDP for the sequence
\begin{equation}
    Z_n := \frac{\sqrt{n}}{b_n}\left(\frac{X_n}{n} - x_\infty\right).
\end{equation}
Our main result is the following theorem.
\begin{theorem}[MDP]\label{thm:main}
    Assume \eqref{bn-condition}. Then the sequence $\{Z_n\}_{n\ge 1}$ satisfies an MDP with speed $b_n^2$ and rate function
    \begin{equation}
        J(x) = \frac{x^2}{2\sigma^2}, \quad x\in\mathbb{R},
    \end{equation}
    where $\sigma^2$ is given by \eqref{eq-sigma2}. That is, for every Borel set $B\subset\mathbb{R}$,
    \begin{align}
        -\inf_{x\in B^\circ} J(x) &\le \liminf_{n\to\infty} \frac{1}{b_n^2}\log\mathbb{P}(Z_n \in B) \notag \\
        &\le \limsup_{n\to\infty} \frac{1}{b_n^2}\log\mathbb{P}(Z_n \in B)
        \le -\inf_{x\in \overline{B}} J(x),
    \end{align}
    where $B^\circ$ and $\overline{B}$ denote the interior and closure of $B$, respectively. In particular, for any $z>0$,
    \begin{equation}
        \lim_{n\to\infty} \frac{1}{b_n^2}\log\mathbb{P}\bigl(|Z_n| \ge z\bigr) = -\frac{z^2}{2\sigma^2}.
    \end{equation}
\end{theorem}

\begin{remark}\label{rem:scale-condition}
    Proposition~\ref{prop:local} below establishes the local estimate under \eqref{bn-weak}. The additional condition \eqref{bn-endpoint} controls the right endpoint. Indeed, Lemma~\ref{lem:exact} below gives
    \[
    \mathbb P(X_n=n-1)=n^{-2}.
    \]
    This is the event that only one initially ignorant individual hears the rumour. On this event the value of $Z_n$ tends to $+\infty$, whereas its logarithmic probability divided by $b_n^2$ is $-2\log n/b_n^2$. Under \eqref{bn-endpoint}, this quantity tends to $-\infty$, making the endpoint atom superexponentially negligible. Conversely, among sequences satisfying \eqref{bn-weak}, if $b_n^2/\log n$ stays bounded away from zero along a subsequence, this atom violates the Gaussian upper bound for $[L,\infty)$ when $L$ is sufficiently large. Thus \eqref{bn-endpoint} is necessary for the full MDP with the quadratic rate $J$.
\end{remark}

The rest of this paper is structured as follows. Section~\ref{sec:numerical} presents a numerical comparison of the CLT, the MDP and the LDP, based on numerical evaluation of the exact final distribution. Section~\ref{sec:prelim} collects several combinatorial preliminaries, including the exact distribution of $X_n$ and the asymptotics of an auxiliary sequence $d_j$, which is formally defined there. In Section~\ref{sec:proof} we prove the MDP: first a local estimate (Proposition~\ref{prop:local}), then the endpoint-sensitive exponential tightness estimate, the upper and lower bounds, and finally the tail bound. A concluding discussion with possible future work is given in Section~\ref{sec:conclusion}. Technical details of the point probability expansion are deferred to Appendix~\ref{appendix}.

%%%%%%%%%%%%%%%%%%%%%%%%%%%%%%%%%%%%%%%%%%%%%%%%%%%%%%%%%%%%%%%%%%%%%%%%%%%%

\section{Numerical comparison of the CLT, MDP and LDP}\label{sec:numerical}

%%%%%%%%%%%%%%%%%%%%%%%%%%%%%%%%%%%%%%%%%%%%%%%%%%%%%%%%%%%%%%%%%%%%%%%%%%%%

We extend the comparison in \cite[Table~1]{Lebensztayn15} from fixed thresholds to the MDP scale. The three limit theorems describe different distances from $x_\infty$ and different notions of accuracy. Let $\Phi$ denote the standard normal distribution function. For thresholds $x_n<x_\infty$ with $\sqrt n(x_\infty-x_n)\to t\ge0$, the CLT gives
\[
 \mathbb P(X_n/n\le x_n)\longrightarrow\Phi(-t/\sigma).
\]
If $x_\infty-x_n=zb_n/\sqrt n$, where $z>0$ is fixed and $b_n$ satisfies the assumptions in Theorem~\ref{thm:main}, then the MDP yields that
\[
 \log\mathbb P(X_n/n\le x_n)
 =-b_n^2J(z)+o(b_n^2),\qquad J(z)=\frac{z^2}{2\sigma^2}.
\]
This covers the intermediate range $n^{-1/2}\ll x_\infty-x_n\ll\sqrt{\log n/n}$. For a fixed $u\in(0,x_\infty)$, the LDP implies that
\[
 \log\mathbb P(X_n/n\le u)=-nh(u)+o(n).
\]
The CLT describes limiting probabilities on the central scale. The MDP and LDP describe logarithmic decay for shrinking and fixed deviations, respectively. The MDP thus gives the asymptotics for thresholds that approach the limiting proportion at an intermediate rate.

For the comparison, choose $b_n=(\log n)^{1/3}$ for $n\ge2$, and set $b_1=1$. This sequence satisfies the assumptions of Theorem~\ref{thm:main}. For $z>0$, define
\[
 \ell_n(z)=\lfloor nx_\infty-zb_n\sqrt n\rfloor
 \quad{\rm and}\quad u_n(z)=\lceil nx_\infty+zb_n\sqrt n\rceil.
\]
All logarithms are natural. The symbols $\lfloor a\rfloor$ and $\lceil a\rceil$ denote the greatest integer not exceeding $a$ and the least integer not less than $a$, respectively. We compare the normalized logarithms
\[
 \lambda_n^{\rm ex}(z)=-\frac{\log\mathbb P(X_n\le\ell_n(z))}{b_n^2},
 \qquad
 \mu_n^{\rm ex}(z)=-\frac{\log\mathbb P(X_n\ge u_n(z))}{b_n^2}
\]
with the Gaussian benchmarks
\begin{align*}
 \lambda_n^{\rm G}(z)&=-\frac1{b_n^2}\log\Phi\!\left(
 \frac{\ell_n(z)+1/2-nx_\infty}{\sigma\sqrt n}\right),\\
 \mu_n^{\rm G}(z)&=-\frac1{b_n^2}\log\left[1-\Phi\!\left(
 \frac{u_n(z)-1/2-nx_\infty}{\sigma\sqrt n}\right)\right].
\end{align*}
The terms $1/2$ and $-1/2$ are continuity corrections. The leading Mills approximation, without a continuity correction, provides the common benchmark
\[
 \lambda_n^{\rm M}(z)=\mu_n^{\rm M}(z)
 =J(z)+\frac1{b_n^2}\left(\log\frac{zb_n}{\sigma}+\frac12\log(2\pi)\right).
\]
It displays the logarithmic contribution of the Gaussian tail prefactor. For cutoffs in the support of $X_n$, also define
\[
 \lambda_n^h(z)=\frac n{b_n^2}h(\ell_n(z)/n)\quad{\rm and}\quad
 \mu_n^h(z)=\frac n{b_n^2}h(u_n(z)/n).
\]
These expressions retain the nonquadratic terms of $h$ near $x_\infty$. The upper benchmark describes the local cost at the threshold. The LDP cost of the whole upper tail is zero, since that tail contains the endpoint at which $H(1)=0$. Theorem~\ref{thm:main}, applied separately to the two tails, yields $\lambda_n^{\rm ex}(z)\to J(z)$ and $\mu_n^{\rm ex}(z)\to J(z)$.

The probabilities are computed numerically by a recursion for the embedded jump chain of the model. Writing $r=n-X(t)$ and $s=Y(t)$, a nonabsorbing state $(r,s)$ moves to $(r,s-1)$ with probability $r/n$ and to $(r+1,s+1)$ with probability $1-r/n$. Propagating these probabilities gives the absorption distribution using only nonnegative terms. At every displayed population size, the computed total probability differs from one by less than $10^{-12}$. The calculation agrees with Lemma~\ref{lem:exact} at small sizes and reproduces all entries of \cite[Table~1]{Lebensztayn15} to the reported precision.

\begin{table}[htbp]
\centering
\small
\setlength{\tabcolsep}{5pt}
\caption{Normalized logarithms of the exact tail probabilities, the Gaussian and leading Mills benchmarks, the local rate benchmarks, and the MDP limit. Here $b_n=(\log n)^{1/3}$; each group fixes $z$ and increases $n$.}
\label{tab:numerical-mdp}
\begin{tabular}{rrrrrrr}
\toprule
\multicolumn{7}{l}{\emph{Lower tail}: $\mathbb P(X_n\le\ell_n(z))$}\\
\midrule
$z$ & $n$ & $\lambda_n^{\rm ex}$ & $\lambda_n^{\rm G}$ & $\lambda_n^{\rm M}$ & $\lambda_n^h$ & $J(z)$ \\
\midrule
0.5 & 500 & 0.977 & 0.933 & 0.897 & 0.509 & 0.458 \\
0.5 & 2000 & 0.930 & 0.906 & 0.860 & 0.493 & 0.458 \\
0.5 & 10000 & 0.872 & 0.860 & 0.826 & 0.471 & 0.458 \\
0.5 & 50000 & 0.835 & 0.830 & 0.800 & 0.464 & 0.458 \\
\addlinespace[2pt]
1.0 & 500 & 2.942 & 2.496 & 2.477 & 2.336 & 1.833 \\
1.0 & 2000 & 2.640 & 2.427 & 2.414 & 2.066 & 1.833 \\
1.0 & 10000 & 2.462 & 2.367 & 2.359 & 1.933 & 1.833 \\
1.0 & 50000 & 2.373 & 2.329 & 2.316 & 1.883 & 1.833 \\
\addlinespace[2pt]
1.5 & 500 & 6.754 & 4.958 & 4.889 & 6.087 & 4.125 \\
1.5 & 2000 & 5.579 & 4.817 & 4.810 & 4.922 & 4.125 \\
1.5 & 10000 & 5.078 & 4.747 & 4.743 & 4.466 & 4.125 \\
1.5 & 50000 & 4.846 & 4.696 & 4.691 & 4.280 & 4.125 \\
\addlinespace[2pt]
2.0 & 500 & 13.287 & 8.165 & 8.182 & 12.637 & 7.333 \\
2.0 & 2000 & 10.035 & 8.099 & 8.093 & 9.323 & 7.333 \\
2.0 & 10000 & 8.829 & 8.025 & 8.016 & 8.159 & 7.333 \\
2.0 & 50000 & 8.315 & 7.957 & 7.958 & 7.694 & 7.333 \\
\midrule
\multicolumn{7}{l}{\emph{Upper tail}: $\mathbb P(X_n\ge u_n(z))$}\\
\midrule
$z$ & $n$ & $\mu_n^{\rm ex}$ & $\mu_n^{\rm G}$ & $\mu_n^{\rm M}$ & $\mu_n^h$ & $J(z)$ \\
\midrule
0.5 & 500 & 0.936 & 0.978 & 0.897 & 0.451 & 0.458 \\
0.5 & 2000 & 0.889 & 0.912 & 0.860 & 0.449 & 0.458 \\
0.5 & 10000 & 0.852 & 0.863 & 0.826 & 0.450 & 0.458 \\
0.5 & 50000 & 0.826 & 0.831 & 0.800 & 0.454 & 0.458 \\
\addlinespace[2pt]
1.0 & 500 & 2.173 & 2.478 & 2.477 & 1.553 & 1.833 \\
1.0 & 2000 & 2.261 & 2.438 & 2.414 & 1.686 & 1.833 \\
1.0 & 10000 & 2.284 & 2.371 & 2.359 & 1.755 & 1.833 \\
1.0 & 50000 & 2.281 & 2.323 & 2.316 & 1.792 & 1.833 \\
\addlinespace[2pt]
1.5 & 500 & 3.575 & 4.932 & 4.889 & 3.244 & 4.125 \\
1.5 & 2000 & 3.868 & 4.833 & 4.810 & 3.604 & 4.125 \\
1.5 & 10000 & 4.132 & 4.754 & 4.743 & 3.857 & 4.125 \\
1.5 & 50000 & 4.338 & 4.698 & 4.691 & 3.994 & 4.125 \\
\addlinespace[2pt]
2.0 & 500 & 3.670 & 8.131 & 8.182 & 5.237 & 7.333 \\
2.0 & 2000 & 3.931 & 8.120 & 8.093 & 6.109 & 7.333 \\
2.0 & 10000 & 4.192 & 8.033 & 8.016 & 6.700 & 7.333 \\
2.0 & 50000 & 4.423 & 7.960 & 7.958 & 7.014 & 7.333 \\
\bottomrule
\end{tabular}
\end{table}

Table~\ref{tab:numerical-mdp} shows that logarithmic corrections remain visible over the displayed range of $n$. At $n=50000$ and $z=1$, the exact lower tail exponent is $2.373$, compared with $J(1)=1.833$. The leading Mills and Gaussian benchmarks are $2.316$ and $2.329$, respectively. These values illustrate the contribution of the Gaussian tail prefactor at this scale. The remaining gap of about $0.06$ reflects the nonquadratic content of $h$. The comparison with the local rate benchmark also changes as the threshold approaches $x_\infty$. For $z=1.5$, the benchmark based on $h$ is closer to the exact lower tail exponent at $n=500$, while the Gaussian benchmark is closer at $n=50000$. The chosen sequence grows slowly, as indicated by $b_{50000}\approx2.21$, and this slow growth is unavoidable because the condition \eqref{bn-endpoint} forces $b_n=o(\sqrt{\log n})$, which means that no admissible scale grows much faster.

The two tails behave differently at the displayed population sizes. For $z=1.5$, the exact lower tail exponent exceeds the Gaussian benchmark, while the upper tail exponent is smaller. The local rate $h$ is asymmetric about $x_\infty$: it climbs to its local maximum $\varrho$ at $1-x_\infty$ and then decreases to $H(1)=0$. The upper tail also includes the endpoint atom. Whenever $u_n(z)\le n-1$, by the identity $\mathbb P(X_n=n-1)=n^{-2}$, we have $\mu_n^{\rm ex}(z)\le{2\log n}/{b_n^2}$.
At $n=50000$ and $z=2$, the bound and the computed upper tail exponent are $4.423535$ and $4.423486$, respectively, to six decimal places. The endpoint atom accounts for more than $99.9\%$ of this tail probability. At the same population size with $z=1.5$, its share is $65.9\%$. For $z=1.5$, the upper tail exponent lies below $J(1.5)$ at the first two population sizes and above it at the last two. Under the condition \eqref{bn-endpoint}, $2\log n/b_n^2\to\infty$, so this atom eventually becomes negligible at the MDP speed.

For comparison with the fixed thresholds in \cite[Table~1]{Lebensztayn15}, take $x=0.15$. The value of $-n^{-1}\log\mathbb P(X_n\le nx)$ is $0.00711$, $0.00624$ and $0.00601$ at $n=2000$, $10000$ and $50000$, respectively. These values decrease towards $h(0.15)=0.00593$. The first agrees with the cited table, whose largest value of $N$ is $2000$; its parameter $N$ is the number of initial ignorants, denoted here by $n$.

%%%%%%%%%%%%%%%%%%%%%%%%%%%%%%%%%%%%%%%%%%%%%%%%%%%%%%%%%%%%%%%%

\section{Combinatorial preliminaries}\label{sec:prelim}

%%%%%%%%%%%%%%%%%%%%%%%%%%%%%%%%%%%%%%%%%%%%%%%%%%%%%%%%%%%%%%%%

In this section we gather the necessary combinatorial tools for the analysis of the final number of ignorants $X_n$. The key ingredients are the exact distribution of $X_n$ derived by Lebensztayn \cite{Lebensztayn15} and the asymptotic behaviour of the auxiliary sequence $d_j$.

%%%%%%%%%%%%%%%%%%%%%%%%%%%%%%%%%%%%%%%%%%%%%%%%%%%%%%%%%%%%%%%%%

\subsection{Exact distribution of \texorpdfstring{$X_n$}{Xn}}

%%%%%%%%%%%%%%%%%%%%%%%%%%%%%%%%%%%%%%%%%%%%%%%%%%%%%%%%%%%%%%%%%

Let $V_n = n - X_n$ denote the number of initially ignorant individuals who heard the rumour by the end of the process. The following exact formula is due to Lebensztayn \cite{Lebensztayn15}; it can also be deduced from the results of Lef\`evre and Picard \cite{LefevrePicard94}.

\begin{lemma}[Exact distribution]\label{lem:exact}
    For $j = 1,\ldots,n$,
    \begin{equation}\label{eq:exact-dist-V}
        \mathbb{P}(V_n = j) = \frac{(n-1)!}{(n-j)!} \frac{d_j}{n^{2j}},
    \end{equation}
    where the sequence $\{d_j\}_{j\ge 1}$ is defined recursively by $d_1=1$ and for $j\ge 2$,
    \begin{equation}\label{eq:dm-recursive}
        d_j = \frac{j^{2j}}{(j-1)!} - \sum_{i=1}^{j-1} \frac{j^{2(j-i)}}{(j-i)!} d_i.
    \end{equation}
    Consequently, for $k = 0,\ldots,n-1$,
    \begin{equation}\label{eq:exact-dist-X}
        \mathbb{P}(X_n = k) = \frac{(n-1)!}{k!} \frac{d_{n-k}}{n^{2(n-k)}}.
    \end{equation}
\end{lemma}

The sequence $d_j$ arises in the enumeration of accessible automata. It counts, up to isomorphism, complete deterministic transition structures on $j$ states over a fixed two-letter alphabet, with a distinguished initial state and no choice of accepting states. Accessibility means that every state can be reached from the initial state. The recurrence \eqref{eq:dm-recursive} goes back to Liskovets \cite{Liskovets69}; see also \cite{BassinoNicaud07}. The first terms are $1,12,216,5248,\ldots$. Its connection with the final-size distribution was established in \cite{Lebensztayn15} through the absorption paths of the embedded chain.

%%%%%%%%%%%%%%%%%%%%%%%%%%%%%%%%%%%%%%%%%%%%%%%%%%%%%%%%%%%%

\subsection{Asymptotics of \texorpdfstring{$d_j$}{d-j}}

%%%%%%%%%%%%%%%%%%%%%%%%%%%%%%%%%%%%%%%%%%%%%%%%%%%%%%%%%%%%%

To analyse moderate deviations, we need precise asymptotics for $d_j$ as $j\to\infty$. The next lemma follows from Lemmas~3.5 and 3.6 of \cite{Lebensztayn15}, which build on results of \cite{BassinoNicaud07,Lebensztayn10}.

\begin{lemma}[Asymptotics of $d_j$]\label{lem:dm-asympt}
    Let $v_\infty = 1 - x_\infty$ and define
    \begin{align*}
        \kappa = 2 - 1/v_\infty,\quad
        \alpha = \sqrt{\frac{1}{2\pi(2v_\infty - 1)}}, \quad
        \beta = \frac{1}{e v_\infty(1 - v_\infty)}.
    \end{align*}
    Then as $j\to\infty$,
    \begin{equation}
        d_j = \alpha \kappa \beta^j j^{j+1/2} \bigl(1 + o(1)\bigr).
    \end{equation}
\end{lemma}

%%%%%%%%%%%%%%%%%%%%%%%%%%%%%%%%%%%%%%%%%%%%%%%%%%%%

\subsection{Point probability estimates}

%%%%%%%%%%%%%%%%%%%%%%%%%%%%%%%%%%%%%%%%%%%%%%%%%%%%

The following lemma provides the key local estimate needed for the MDP. Its proof uses Taylor expansions; the detailed calculations are deferred to Appendix~\ref{appendix}. A crucial point for the subsequent proofs is that this estimate holds uniformly for $z$ in any compact interval.

\begin{lemma}[Point probability estimate]\label{lem:point}
    Assume \eqref{bn-weak} and fix a compact interval $K\subset\mathbb R$. For $z\in K$, define
    \[
    k_n(z) = \bigl\lfloor n x_\infty + z b_n \sqrt{n} \bigr\rfloor,
    \]
    where $\lfloor u\rfloor$ denotes the greatest integer not exceeding $u$. Then, as $n\to\infty$,
    \[
    \log \mathbb{P}(X_n = k_n(z)) = -\frac12 \log n - \frac{z^2 b_n^2}{2\sigma^2} + o(b_n^2),
    \]
    where the $o(b_n^2)$ term is uniform for $z \in K$. Equivalently,
    \[
    \mathbb{P}(X_n = k_n(z)) = n^{-1/2} \exp\left(-\frac{z^2 b_n^2}{2\sigma^2} + o(b_n^2)\right)
    \]
  holds uniformly for $z \in K$.
\end{lemma}

For the uniformity in Lemma~\ref{lem:point}, write the floor correction as $\delta_n(z)\in[-1,0]$. On a fixed compact interval $K$, both $k_n(z)$ and $n-k_n(z)$ exceed a positive multiple of $n$. The Stirling errors are therefore $O_K(n^{-1})$. If $r_m=o(1)$ is the relative error in the asymptotic formula for $d_m$, then $\sup_{m\ge c_Kn}|r_m|\to0$. The cubic Taylor remainder is $O_K(b_n^3/\sqrt n)=o(b_n^2)$, and the floor changes the quadratic term by $O_K(b_n/\sqrt n+n^{-1})=o(1)$. Appendix~\ref{appendix} gives the termwise calculation.

%%%%%%%%%%%%%%%%%%%%%%%%%%%%%%%%%%%%%%%%%%%%%%%%%%%%%%%%%%%%%%

\section{Proof of the MDP}\label{sec:proof}

%%%%%%%%%%%%%%%%%%%%%%%%%%%%%%%%%%%%%%%%%%%%%%%%%%%%%%%%%%%%%%

With the combinatorial estimates established in the previous section, we now proceed to the proof of Theorem~\ref{thm:main}. The strategy is standard for MDPs: we first obtain a local estimate for probabilities of small intervals, then extend to arbitrary Borel sets via approximation arguments.

%%%%%%%%%%%%%%%%%%%%%%%%%%%%%%%%%%%%%%%%%%%%%%%%%%%%%%%%%%

\subsection{Local moderate deviation estimate}

%%%%%%%%%%%%%%%%%%%%%%%%%%%%%%%%%%%%%%%%%%%%%%%%%%%%%%%%%%

The following estimate gives the MDP bounds on bounded intervals. It will be used with a finite covering argument for closed sets.

\begin{proposition}[Local MDP]\label{prop:local}
    Let $I=(a,b)$ be a nonempty bounded open interval. Assume \eqref{bn-weak}. Then
    \begin{align}
        \liminf_{n\to\infty}\frac1{b_n^2}\log \mathbb{P}(Z_n\in I)
        &\ge -\inf_{z\in I} \frac{z^2}{2\sigma^2}, \label{eq:local-lower-fixed}\\
        \limsup_{n\to\infty}\frac1{b_n^2}\log \mathbb{P}(Z_n\in I)
        &\le -\inf_{z\in \overline I} \frac{z^2}{2\sigma^2}.
    \end{align}
    In particular, for any $x\in\mathbb{R}$,
    \begin{equation}\label{eq:local-MDP}
        \lim_{\delta\downarrow 0}\lim_{n\to\infty}
        \frac{1}{b_n^2}\log \mathbb{P}\bigl(Z_n \in (x-\delta, x+\delta)\bigr) = -\frac{x^2}{2\sigma^2}.
    \end{equation}
\end{proposition}

\begin{proof}
    Write $q_n=b_n\sqrt n$ and $z_{n,k}=(k-nx_\infty)/q_n$. Then $k_n(z_{n,k})=k$, and the event $\{Z_n\in I\}$ is the disjoint union of $\{X_n=k\}$ over those integers with $z_{n,k}\in I$. Since $q_n/n\to0$, every such integer belongs to the support $\{0,\ldots,n-1\}$ for all sufficiently large $n$. There are at most $(b-a)q_n+2$ such integers.

    \medskip
    \noindent\textit{Upper bound.}
    Let $K=[a-1,b+1]$. For all large $n$, every point contributing to $\{Z_n\in I\}$ can be written as $k=\lfloor n x_\infty+z_k b_n\sqrt n\rfloor$ with $z_k=z_{n,k}\in I\subset K$. Lemma~\ref{lem:point}, uniformly on $K$, gives that for every $\epsilon>0$ and all large $n$,
    \[
        \mathbb{P}(X_n=k)
        \le n^{-1/2}\exp\left\{-b_n^2\left(\frac{z_k^2}{2\sigma^2}-\epsilon\right)\right\}.
    \]
    Since the admissible $z_k$ lie in $I$, the preceding bound yields
    \[
        \mathbb{P}(Z_n\in I)
        \le C b_n \exp\left\{-b_n^2\left(\inf_{z\in \overline I}\frac{z^2}{2\sigma^2}-\epsilon\right)\right\}
    \]
    for all large $n$ and a constant $C$ depending on $I$ only. After division by $b_n^2$ and passage to the limit, the term $\log(Cb_n)/b_n^2$ vanishes, and letting $\epsilon\downarrow0$ gives the upper bound.

    \medskip
    \noindent\textit{Lower bound.}
    Fix $z_0\in I$ and choose $\rho>0$ such that $[z_0-\rho,z_0+\rho]\subset I$.
    The integers satisfying $z_{n,k}\in[z_0-\rho,z_0+\rho]$ lie in the interval
    \[
        [nx_\infty+(z_0-\rho)q_n,\,nx_\infty+(z_0+\rho)q_n].
    \]
    Its length is $2\rho q_n$, so it contains at least $2\rho q_n-1\ge\rho q_n=\rho b_n\sqrt n$ integers for all sufficiently large $n$. Each contributes to $\{Z_n\in I\}$, and $k_n(z_{n,k})=k$ makes Lemma~\ref{lem:point} applicable to every one of them. By the uniform lower bound in Lemma~\ref{lem:point}, for every $\epsilon>0$ and all large $n$ each such integer has probability at least
    \[
        n^{-1/2}\exp\left\{-b_n^2\left(\sup_{|z-z_0|\le \rho}\frac{z^2}{2\sigma^2}+\epsilon\right)\right\}.
    \]
    Hence
    \[
        \mathbb{P}(Z_n\in I)
        \ge \rho b_n\exp\left\{-b_n^2\left(\sup_{|z-z_0|\le \rho}\frac{z^2}{2\sigma^2}+\epsilon\right)\right\}.
    \]
    Letting $n\to\infty$, then $\epsilon\downarrow0$, and finally $\rho\downarrow0$, we get
    \[
        \liminf_{n\to\infty}\frac1{b_n^2}\log\mathbb{P}(Z_n\in I)
        \ge -\frac{z_0^2}{2\sigma^2}.
    \]
    Since $z_0\in I$ is arbitrary, \eqref{eq:local-lower-fixed} follows. Since $J$ is continuous, $\inf_IJ=\inf_{\overline I}J$, so the inner limit exists for each fixed interval. Taking $I=(x-\delta,x+\delta)$ and then $\delta\downarrow0$ proves \eqref{eq:local-MDP}.
\end{proof}

%%%%%%%%%%%%%%%%%%%%%%%%%%%%%%%%%%%%%%%%%%%%%%%%%%%%%%%%%%%%%%%%%

\subsection{An endpoint-sensitive exponential tightness estimate}

%%%%%%%%%%%%%%%%%%%%%%%%%%%%%%%%%%%%%%%%%%%%%%%%%%%%%%%%%%%%%%%%%

The next lemma supplies the tightness estimate needed for the full upper bound. The function $h$ tends to zero at the right endpoint, where extinction after a few transmissions gives a polynomial probability contribution. We estimate this contribution directly from the exact distribution of $X_n$.

\begin{lemma}[Exponential tightness]\label{lem:exp-tight}
    Assume \eqref{bn-condition}. Then for every $C>0$ there exists $L>0$ such that
    \[
        \limsup_{n\to\infty}\frac1{b_n^2}\log\mathbb{P}(|Z_n|>L)\le -C.
    \]
\end{lemma}

\begin{proof}
    We first obtain a uniform pointwise upper bound in a fixed neighbourhood of $x_\infty$. From the exact formula \eqref{eq:exact-dist-X}, Lemma~\ref{lem:dm-asympt}, and Stirling's formula, there exist $r>0$ and $A<\infty$ such that, for all large $n$ and all $0\le k\le n-1$ with $|k/n-x_\infty|\le r$,
    \begin{equation}\label{eq:global-point-upper}
        \mathbb{P}(X_n=k)\le A n^{-1/2} \exp\{-n h(k/n)\}.
    \end{equation}
    Indeed, the same Stirling calculation, before expansion about $x_\infty$, gives
    \[
    \log\mathbb P(X_n=k)
    =-\tfrac12\log n-nh(u)+\log(\alpha\kappa)
       +\tfrac12\log\frac{1-u}{u}+o(1),\qquad u=k/n,
    \]
    uniformly on a fixed compact neighbourhood of $x_\infty$. Here $\log\beta-1=-\varrho$. The remaining terms are uniformly bounded, which proves \eqref{eq:global-point-upper}.

    Since $h(x_\infty)=h'(x_\infty)=0$ and $h''(x_\infty)=1/\sigma^2$, reducing $r$ if necessary gives
    \begin{equation}
        h(u)\ge \frac{(u-x_\infty)^2}{4\sigma^2},
        \qquad |u-x_\infty|\le r.
    \end{equation}
    Hence, for $L>0$ and all large $n$,
    \begin{align*}
        \mathbb{P}\left(L<|Z_n|\le r\frac{\sqrt n}{b_n}\right)
        &\le A n^{-1/2}\sum_{|k-nx_\infty|>L b_n\sqrt n}
        \exp\left\{-\frac{(k-nx_\infty)^2}{4\sigma^2 n}\right\} \\
        &\le A' \exp\left\{-\frac{L^2 b_n^2}{8\sigma^2}\right\},
    \end{align*}
    To justify the last inequality, set $t_n=Lb_n\sqrt n$. For $|k-nx_\infty|>t_n$,
    \[
    e^{-(k-nx_\infty)^2/(4\sigma^2 n)}
    \le e^{-L^2b_n^2/(8\sigma^2)}
       e^{-(k-nx_\infty)^2/(8\sigma^2 n)}.
    \]
    Comparing the two monotone halves of the remaining Gaussian sum with their integrals gives, uniformly in its centre,
    \[
    \sum_{k\in\mathbb Z}e^{-(k-nx_\infty)^2/(8\sigma^2 n)}
    \le 2+\int_{-\infty}^{\infty}e^{-s^2/(8\sigma^2 n)}\,ds
    \le C_0\sqrt n.
    \]
    This cancels the factor $n^{-1/2}$ in the point probability bound.
    Therefore
    \begin{equation}\label{eq:moderate-tail}
        \limsup_{n\to\infty}\frac1{b_n^2}
        \log \mathbb{P}\left(L<|Z_n|\le r\frac{\sqrt n}{b_n}\right)
        \le -\frac{L^2}{8\sigma^2}.
    \end{equation}

    It remains to control the fixed-distance part $|X_n/n-x_\infty|>r$.  Choose $\delta>0$ so small that $\delta<v_\infty-r$.  On the set
    \[
        \{X_n/n\le x_\infty-r\}
        \cup
        \{x_\infty+r\le X_n/n\le 1-\delta\},
    \]
    the LDP upper bound of Lebensztayn \cite{Lebensztayn15} gives an estimate of order $\exp\{-cn\}$ for some $c>0$; this is superexponential on the $b_n^2$ scale because \eqref{bn-weak} implies $b_n^2/n\to0$.

    For the endpoint layer $X_n/n>1-\delta$, equivalently $V_n<\delta n$, we use the exact distribution. As explained in Remark~\ref{rem:scale-condition}, condition \eqref{bn-endpoint} makes the resulting polynomial bound superexponential on the $b_n^2$ scale.  From Lemma~\ref{lem:dm-asympt} there are constants $B,D<\infty$ such that $d_j\le B D^j j^{j+1/2}$ for all $j\ge1$. Using \eqref{eq:exact-dist-V}, for $1\le j\le n$,
    \[
        \mathbb{P}(V_n=j)
        \le \frac{d_j}{n^{j+1}}
        \le \frac{B j^{1/2}}{n}\left(\frac{D j}{n}\right)^j.
    \]
    Decrease $\delta$ further, if necessary, so that $D\delta<1/2$.  Summing the preceding bound over $1\le j\le \delta n$ gives
    \begin{equation}\label{eq:endpoint-poly}
        \mathbb{P}(V_n\le \delta n)\le\frac{B}{n}\sum_{j\ge1}j^{1/2}2^{-j}\le\frac{B'}{n}
    \end{equation}
    for all large $n$. Because $b_n^2/\log n\to0$, the bound \eqref{eq:endpoint-poly} is superexponential on the $b_n^2$ scale. Combining this endpoint estimate with \eqref{eq:moderate-tail} and the fixed-distance LDP bound proves the lemma after choosing $L$ so large that $L^2/(8\sigma^2)>C$.
\end{proof}

%%%%%%%%%%%%%%%%%%%%%%%%%%%%%%%%%%%%%%%%%%%%%%%%%%%%%

\subsection{Upper bound}

%%%%%%%%%%%%%%%%%%%%%%%%%%%%%%%%%%%%%%%%%%%%%%%%%%%%%

We now prove the upper bound in Theorem \ref{thm:main}: for any closed set $F \subset \mathbb{R}$,
\begin{equation}\label{upper-bound}
    \limsup_{n\to\infty} \frac{1}{b_n^2}\log \mathbb{P}(Z_n \in F)
    \le -\inf_{x \in F} J(x).
\end{equation}

\begin{proof}
    If $F=\varnothing$ there is nothing to prove. Otherwise, let $\gamma=\inf_{x\in F}J(x)<\infty$. Fix $\epsilon>0$. By Lemma~\ref{lem:exp-tight}, choose $L>0$ so large that
    \[
        \limsup_{n\to\infty}\frac1{b_n^2}\log\mathbb{P}(|Z_n|>L)
        \le -\gamma-\epsilon.
    \]
    Enlarging $L$ if necessary, assume that $F\cap[-L,L]$ is nonempty.

    The compact set $F_L=F\cap[-L,L]$ can be covered by finitely many open intervals $I_1,\ldots,I_N$ of radius $\delta$. Applying Proposition~\ref{prop:local} to each interval and using the union bound,
    \[
        \limsup_{n\to\infty}\frac1{b_n^2}\log\mathbb{P}(Z_n\in F_L)
        \le -\min_{1\le j\le N}\inf_{z\in \overline{I_j}} J(z).
    \]
    Choose the cover so that every $I_j$ is centred at a point of $F_L$. Since $J$ is uniformly continuous on $[-L-1,L+1]$, letting $\delta\downarrow0$ yields
    \[
        \limsup_{n\to\infty}\frac1{b_n^2}\log\mathbb{P}(Z_n\in F_L)
        \le -\inf_{z\in F_L}J(z).
    \]
    Consequently,
    \begin{align*}
        \limsup_{n\to\infty}\frac1{b_n^2}\log\mathbb{P}(Z_n\in F)
        \le \max\left\{-\inf_{z\in F_L}J(z),\,-\gamma-\epsilon\right\}\le -\gamma.
    \end{align*}
    This proves \eqref{upper-bound}.
\end{proof}

%%%%%%%%%%%%%%%%%%%%%%%%%%%%%%%%%%%%%%%%%%%%%%%%%

\subsection{Lower bound}

%%%%%%%%%%%%%%%%%%%%%%%%%%%%%%%%%%%%%%%%%%%%%%%%

We now prove the lower bound in Theorem \ref{thm:main}: for any open set $G \subset \mathbb{R}$,
\begin{equation}\label{lower-bound}
    \liminf_{n\to\infty} \frac{1}{b_n^2}\log \mathbb{P}(Z_n \in G)
    \ge -\inf_{x \in G} J(x).
\end{equation}

\begin{proof}
    If $G=\varnothing$ there is nothing to prove. Fix $x_0\in G$. Since $G$ is open, there exists $\delta>0$ such that
    $(x_0-\delta,x_0+\delta)\subset G$. Proposition~\ref{prop:local} gives
    \[
        \liminf_{n\to\infty}\frac1{b_n^2}
        \log\mathbb{P}(Z_n\in G)
        \ge
        \liminf_{n\to\infty}\frac1{b_n^2}
        \log\mathbb{P}\bigl(Z_n\in(x_0-\delta,x_0+\delta)\bigr)
        \ge -\inf_{|z-x_0|<\delta}J(z).
    \]
    Letting $\delta\downarrow0$ gives
    \[
        \liminf_{n\to\infty}\frac1{b_n^2}
        \log\mathbb{P}(Z_n\in G)\ge -J(x_0).
    \]
    Finally take the supremum over $x_0\in G$, equivalently the negative of the infimum of $J$ over $G$, to obtain \eqref{lower-bound}.
\end{proof}

%%%%%%%%%%%%%%%%%%%%%%%%%%%%%%%%%%%%%%%%%%%%%%%%%%%%%%

\subsection{Proof of the tail bound}

%%%%%%%%%%%%%%%%%%%%%%%%%%%%%%%%%%%%%%%%%%%%%%%%%%%%%%

The tail bound follows immediately from the MDP by taking $B = (-\infty, -z] \cup [z, \infty)$. Since $J$ is continuous and even, $\inf_{|x| \ge z} J(x) = J(z) = z^2/(2\sigma^2)$. The upper bound applies to $B$, and the lower bound to its interior, which has the same infimum of $J$. Hence
\[
    \lim_{n\to\infty} \frac{1}{b_n^2} \log \mathbb{P}(|Z_n| \ge z) = -\frac{z^2}{2\sigma^2}.
\]
This completes the proof of Theorem \ref{thm:main}.

%%%%%%%%%%%%%%%%%%%%%%%%%%%%%%%%%%%%%%%%%%%%%%%%%%%%%%%%%%

\section{Concluding remarks}\label{sec:conclusion}

%%%%%%%%%%%%%%%%%%%%%%%%%%%%%%%%%%%%%%%%%%%%%%%%%%%%%%%%%%

In this note we have established the MDP for the final proportion of ignorants in the classical Maki--Thompson rumour model. Two directions for future research seem particularly natural. First, the assumption of homogeneous mixing is an idealisation. It would be natural to extend the analysis to sparse random networks, such as Erd\H{o}s-R\'{e}nyi graphs or configuration models. Recent work (\cite{APRT,DPZMV,FJMRM2022}) indicates that network topology can induce qualitatively different phase transitions in information spreading, and the interplay between rare events and graph heterogeneity may lead to nonstandard moderate deviation behaviour. For heterogeneous mixing with finitely many types, one could track the counts of each type and seek a path-space MDP, followed by an analysis of the extinction time. The main difficulty is controlling the process near absorption and the continuity of the extinction map. On sparse networks, aggregate counts are generally not Markovian; degrees and previously explored contacts must also be tracked. Random graph fluctuations may contribute to annealed deviations, whereas quenched deviations condition on the graph. These features prevent direct use of the exact-distribution argument; see also \cite{DraiefMassoulie10,DiazBaccaEtAl26}. Second, one could investigate moderate deviations for variants of the model, including those with stochastic stifling mechanisms or multiple initial spreaders \cite{LebensztaynMachadoRodriguez11,LebensztaynMachadoRodriguezRandom11,LeRo2025}. For such variants closed-form final-size distributions are typically unavailable, so the combinatorial approach used here may not transfer directly. The corresponding LDPs and MDPs are natural questions for further study.

%%%%%%%%%%%%%%%%%%%%%%%%%%%%%%%%%%%%%%%%%%%%%%%%%%%%%%%%%%%%%%%%%%%%%

\appendix

\section{Detailed expansion for Lemma \ref{lem:point}}\label{appendix}

%%%%%%%%%%%%%%%%%%%%%%%%%%%%%%%%%%%%%%%%%%%%%%%%%%%%%%%%%%%%%%%%%%%%%%

In this appendix we give the detailed termwise Taylor expansion for
the point probability estimate. Throughout, assume \eqref{bn-weak}
and fix a compact interval $K\subset\mathbb R$. Write
$M=\sup_{z\in K}|z|$, $v_\infty=1-x_\infty$, and
\[
 k_n=nx_\infty+\eta_n,\qquad m_n=nv_\infty-\eta_n,\qquad
 \eta_n=zb_n\sqrt n+\delta_n(z),
\]

where
\[
 \delta_n(z)=\lfloor nx_\infty+zb_n\sqrt n\rfloor
              -(nx_\infty+zb_n\sqrt n)\in[-1,0].
\]
Thus $|\eta_n|\le Mb_n\sqrt n+1$ and $\eta_n/n=o(1)$ uniformly
for $z\in K$. There exists $c_K>0$ such that $k_n,m_n\ge c_Kn$
for all sufficiently large $n$ and every $z\in K$. All estimates below
are uniform on $K$; constants may depend on $K$.

By Lemma \ref{lem:exact} and Lemma \ref{lem:dm-asympt},
\begin{align*}
 \mathbb P(X_n=k_n)&=\frac{(n-1)!}{k_n!}\,
                         \frac{d_{m_n}}{n^{2m_n}},\\
 d_{m_n}&=\alpha\kappa\beta^{m_n}m_n^{m_n+1/2}(1+r_{m_n}),
 \qquad r_m\longrightarrow0.
\end{align*}
Stirling's formula gives, for any large enough positive integer $\ell$,
\[
\log \ell! = \ell \log \ell - \ell + \frac12 \log(2\pi \ell) + O(\ell^{-1}).
\]
Applying this to $(n-1)!$ and $k_n!$, we obtain
\begin{align*}
\log \mathbb{P}(X_n=k_n) &= \frac12 \log\frac{n-1}{k_n} + (n-1)\log(n-1) - k_n \log k_n \\
&\quad + m_n \log\beta + m_n \log m_n + \frac12 \log m_n - 2 m_n \log n \\
&\quad + k_n - (n-1) + \log(\alpha\kappa)+\varepsilon_n(z),
\end{align*}
where
\[
 \sup_{z\in K}|\varepsilon_n(z)|
 \le C_Kn^{-1}+2\sup_{m\ge c_Kn}|r_m|=o(1).
\]
Indeed, the two Stirling remainders are $O_K(n^{-1})$, while
$\sup_{m\ge c_Kn}|r_m|\to0$. For large $n$, the latter supremum
is at most $1/2$, so $|\log(1+r_{m_n})|\le2|r_{m_n}|$.

We now expand each term. The estimate
\[
 \log(1+\varepsilon)=\varepsilon-\varepsilon^2/2+O(|\varepsilon|^3)
\]
holds uniformly for $|\varepsilon|\le1/2$. This restriction applies
to $\eta_n/(nx_\infty)$ and $\eta_n/(nv_\infty)$ for all sufficiently
large $n$. After multiplication by a quantity of order $n$, the cubic
remainders are $O_K(|\eta_n|^3/n^2)=O_K(b_n^3/\sqrt n)=o(b_n^2)$.

\begin{itemize}
\item $\displaystyle \frac12 \log\frac{n-1}{k_n}$:
\begin{align*}
\log(n-1) &= \log n - \frac1n - \frac{1}{2n^2} + O\Bigl(\frac1{n^3}\Bigr), \\
\log k_n &= \log n + \log x_\infty + \frac{\eta_n}{n x_\infty}
- \frac{\eta_n^2}{2 n^2 x_\infty^2} + O\Bigl(\frac{|\eta_n|^3}{n^3}\Bigr).
\end{align*}

Hence
\[
\frac12 \log\frac{n-1}{k_n} = -\frac12 \log x_\infty - \frac{\eta_n}{2n x_\infty}
+ \frac{\eta_n^2}{4 n^2 x_\infty^2} - \frac{1}{2n} + o(1).
\]

\item $\displaystyle (n-1)\log(n-1)$:
\[
(n-1)\log(n-1) = n\log n - \log n - 1 + \frac{1}{2n} + o\Bigl(\frac1n\Bigr).
\]

\item $\displaystyle -k_n \log k_n$:
\begin{align*}
k_n \log k_n &= (n x_\infty + \eta_n)\Bigl(\log n + \log x_\infty
+ \frac{\eta_n}{n x_\infty} - \frac{\eta_n^2}{2 n^2 x_\infty^2}
+ O\Bigl(\frac{|\eta_n|^3}{n^3}\Bigr)\Bigr) \\
&= n x_\infty \log n + n x_\infty \log x_\infty + \eta_n \log n + \eta_n \log x_\infty \\
&\quad + \eta_n + \frac{\eta_n^2}{n x_\infty} - \frac{\eta_n^2}{2n x_\infty} + o(b_n^2).
\end{align*}

Thus
\[
-k_n \log k_n = -n x_\infty \log n - n x_\infty \log x_\infty - \eta_n \log n
- \eta_n \log x_\infty - \eta_n - \frac{\eta_n^2}{2n x_\infty} + o(b_n^2).
\]

\item $\displaystyle m_n \log\beta$:
\[
m_n \log\beta = n v_\infty \log\beta - \eta_n \log\beta.
\]

\item $\displaystyle m_n \log m_n$:
First expand $\log m_n$:
\[
\log m_n = \log n + \log v_\infty + \log\Bigl(1 - \frac{\eta_n}{n v_\infty}\Bigr).
\]
Using $\log(1-\varepsilon) = -\varepsilon - \frac12 \varepsilon^2 + O(\varepsilon^3)$
with $\varepsilon = \frac{\eta_n}{n v_\infty}$,
\[
\log m_n = \log n + \log v_\infty - \frac{\eta_n}{n v_\infty}
- \frac{\eta_n^2}{2 n^2 v_\infty^2} + O\Bigl(\frac{|\eta_n|^3}{n^3}\Bigr).
\]

Then
\begin{align*}
m_n \log m_n &= (n v_\infty - \eta_n)\Bigl(\log n + \log v_\infty
- \frac{\eta_n}{n v_\infty} - \frac{\eta_n^2}{2 n^2 v_\infty^2}\Bigr) + o(b_n^2) \\
&= n v_\infty \log n + n v_\infty \log v_\infty - \eta_n - \frac{\eta_n^2}{2n v_\infty} \\
&\quad - \eta_n \log n - \eta_n \log v_\infty + \frac{\eta_n^2}{n v_\infty} + o(b_n^2) \\
&= n v_\infty \log n + n v_\infty \log v_\infty - \eta_n \log n - \eta_n \log v_\infty \\
&\quad - \eta_n + \frac{\eta_n^2}{2n v_\infty} + o(b_n^2).
\end{align*}

\item $\displaystyle \frac12 \log m_n$:
\[
\frac12 \log m_n = \frac12 \log n + \frac12 \log v_\infty
- \frac{\eta_n}{2n v_\infty} - \frac{\eta_n^2}{4 n^2 v_\infty^2} + o(1).
\]

\item $\displaystyle -2 m_n \log n$:
\[
-2 m_n \log n = -2(n v_\infty - \eta_n) \log n = -2 n v_\infty \log n + 2 \eta_n \log n.
\]

\item $\displaystyle k_n - (n-1)$:
\[
k_n - (n-1) = n x_\infty + \eta_n - n + 1 = -n v_\infty + \eta_n + 1.
\]
\end{itemize}

Now sum all contributions, grouping terms of the same type:

\begin{itemize}

\item \textbf{$n \log n$ terms:}
\begin{align*}
&(n-1)\log(n-1): && + n\log n, \\
&-k_n\log k_n: && - n x_\infty \log n, \\
&m_n\log m_n: && + n v_\infty \log n, \\
&-2 m_n \log n: && -2 n v_\infty \log n.
\end{align*}
The sum is
$$n \log n\,(1 - x_\infty + v_\infty - 2v_\infty) = n \log n\,(1 - x_\infty - v_\infty) = 0.$$

\item \textbf{$\eta_n \log n$ terms:}
\begin{align*}
&-k_n\log k_n: && -\eta_n \log n, \\
&m_n\log m_n: && -\eta_n \log n, \\
&-2 m_n \log n: && +2 \eta_n \log n.
\end{align*}
The sum is $0$.

\item \textbf{$\log n$ terms (without $\eta_n$):}
\begin{align*}
&(n-1)\log(n-1): && -\log n, \\
&\frac12 \log m_n: && +\frac12 \log n.
\end{align*}
The contribution is $-\frac12 \log n$.

\item \textbf{Linear $n$ terms (no $\log$):}
\begin{align*}
&-k_n\log k_n: && -n x_\infty \log x_\infty, \\
&m_n\log\beta: && +n v_\infty \log\beta, \\
&m_n\log m_n: && +n v_\infty \log v_\infty, \\
&k_n-(n-1): && -n v_\infty.
\end{align*}
Sum: $n\bigl(-x_\infty \log x_\infty + v_\infty \log\beta + v_\infty \log v_\infty - v_\infty\bigr)$.
The cancellation of the order-$n$ term uses both the definition of $\beta$ and the fixed-point equation for $x_\infty$. Using $\beta = \frac{1}{e v_\infty x_\infty}$, we get
$$
\log\beta = -1 - \log v_\infty - \log x_\infty,
$$
so the sum becomes
\begin{align*}
n\bigl(&-x_\infty \log x_\infty - v_\infty - v_\infty \log v_\infty - v_\infty \log x_\infty \\
&\quad + v_\infty \log v_\infty - v_\infty\bigr)
= n\bigl(-(x_\infty + v_\infty) \log x_\infty - 2v_\infty\bigr) \\
&= n\bigl(-\log x_\infty - 2v_\infty\bigr).
\end{align*}
Notice that the fixed-point equation $x_\infty e^{2(1-x_\infty)} = 1$ gives $\log x_\infty + 2v_\infty = 0$, so the sum vanishes.

\item \textbf{Linear $\eta_n$ terms:}
From the various expansions:
\begin{align*}
&\frac12\log\frac{n-1}{k_n}: && -\frac{\eta_n}{2n x_\infty} \quad (\text{negligible, } o(1)), \\
&\frac12\log m_n: && -\frac{\eta_n}{2n v_\infty} \quad (\text{negligible, } o(1)), \\
&-k_n\log k_n: && -\eta_n \log x_\infty - \eta_n, \\
&m_n\log\beta: && -\eta_n \log\beta, \\
&m_n\log m_n: && -\eta_n \log v_\infty - \eta_n, \\
&k_n-(n-1): && +\eta_n.
\end{align*}
The two negligible terms are $O_K(b_n/\sqrt{n})=o(1)$. The remaining sum is
\[
\eta_n\bigl(-\log x_\infty - 1 - \log\beta - \log v_\infty - 1 + 1\bigr)
= \eta_n\bigl(-\log x_\infty - \log\beta - \log v_\infty - 1\bigr).
\]
Substituting $\log\beta$ gives zero.

\item \textbf{Quadratic $\eta_n^2$ terms:}
\begin{align*}
&\frac12\log\frac{n-1}{k_n}: && \frac{\eta_n^2}{4 n^2 x_\infty^2} \quad (\text{negligible, } O(b_n^2/n) \to 0), \\
&-k_n\log k_n: && -\frac{\eta_n^2}{2n x_\infty}, \\
&m_n\log m_n: && \frac{\eta_n^2}{2n v_\infty}, \\
&\frac12\log m_n: && -\frac{\eta_n^2}{4 n^2 v_\infty^2} \quad (\text{negligible}).
\end{align*}
The reciprocal fractions in the next display identify the variance in the CLT. Thus the dominant quadratic contribution is
\[
\frac{\eta_n^2}{2n}\left(-\frac1{x_\infty} + \frac1{v_\infty}\right)
= \frac{\eta_n^2}{2n} \cdot \frac{x_\infty - v_\infty}{x_\infty v_\infty}
= -\frac{\eta_n^2}{2n} \cdot \frac{1-2x_\infty}{x_\infty(1-x_\infty)}.
\]
Now write $\eta_n^2 = z^2 b_n^2 n + 2 z b_n \sqrt{n} \delta_n + \delta_n^2$. Since $\delta_n = O(1)$ and $b_n/\sqrt{n} \to 0$,
\[
\frac{\eta_n^2}{2n} = \frac{z^2 b_n^2}{2} + \frac{z b_n \delta_n}{\sqrt{n}} + \frac{\delta_n^2}{2n}
= \frac{z^2 b_n^2}{2} + o(b_n^2),
\]
because $\frac{b_n/\sqrt{n}}{b_n^2} = \frac{1}{b_n \sqrt{n}} \to 0$ and $\frac{1/n}{b_n^2} = \frac{1}{n b_n^2} \to 0$.
Hence the quadratic contribution becomes
\[
-\frac{z^2 b_n^2}{2\sigma^2} + o(b_n^2),
\]
where
\[
\sigma^2 = \frac{x_\infty(1-x_\infty)}{1-2x_\infty}.
\]
\end{itemize}

Combining the preceding expansions gives
\[
\log \mathbb{P}(X_n = k_n) = -\frac12 \log n - \frac{z^2 b_n^2}{2\sigma^2} + o(b_n^2).
\]

\medskip

\noindent\textit{Uniformity in $z$.}
The preceding estimates account for every error term. The Stirling
remainders are $O_K(n^{-1})$, and the logarithm of the relative error
in $d_{m_n}$ is at most $2\sup_{m\ge c_Kn}|r_m|=o(1)$. The logarithmic
prefactors have errors $O_K(b_n/\sqrt n+b_n^2/n+n^{-1})$. The terms
$k_n\log k_n$ and $m_n\log m_n$ have Taylor errors of order
$O_K(b_n^3/\sqrt n)$. Finally, the floor changes the quadratic term by
$O_K(b_n/\sqrt n+n^{-1})$. In particular,
\[
 \sup_{z\in K}\left|
 \log\mathbb P(X_n=k_n(z))+\frac12\log n
             +\frac{z^2b_n^2}{2\sigma^2}\right|
 \le C_K\left(1+\frac{b_n^3}{\sqrt n}\right)+o(1).
\]
Dividing by $b_n^2$ gives a quantity bounded by
$C_K(b_n^{-2}+b_n/\sqrt n)+o(b_n^{-2})$, which tends to zero.
This proves the uniform assertion of Lemma~\ref{lem:point}.

%%%%%%%%%%%%%%%%%%%%%%%%%%%%%%%%%%%%%%%%%%%%%%

\section*{Acknowledgements}

%%%%%%%%%%%%%%%%%%%%%%%%%%%%%%%%%%%%%%%%%%%%%%

We are grateful to the Editor and the Referee for their careful reading and helpful suggestions, especially those concerning the numerical comparisons in Section \ref{sec:numerical}, which have improved the clarity and presentation of the whole paper. Shaochen Wang is partially supported by the National Natural Science Foundation of China (NSFC) Grant 11801181 and Guangyu Yang is partially supported by the Postgraduate Education Reform and Quality Improvement Project of Henan Province Grant YJS2027ALPY03.

%%%%%%%%%%%%%%%%%%%%%%%%%%%%%%%%%%%%%%%%%%%%%%%%%

\section*{Declarations}

%%%%%%%%%%%%%%%%%%%%%%%%%%%%%%%%%%%%%%%%%%%%%%%%%

\noindent\textbf{Generative AI use.}\ \ During the revision of this manuscript, the authors used ChatGPT (OpenAI, [GPT-6 Astra]) to assist with checking mathematical proofs, improving the English language, and editing the code used for numerical computations. These uses aimed to help identify potential errors and improve the clarity of the presentation and the reliability of the code. The authors independently reviewed and verified all AI-assisted content, including the mathematical arguments, code, and numerical results, and take full responsibility for the accuracy and integrity of the manuscript.
\medskip

\noindent\textbf{Competing interests.}\ \ The authors declare no competing interests.

%%%%%%%%%%%%%%%%%%%%%%%%%%%%%%%%%%%%%%%%%%%%%%%%%%%%%

\end{document}